\documentclass[11pt]{amsart}
\usepackage[utf8]{inputenc}
\usepackage[T1]{fontenc}
\usepackage{a4wide}
\usepackage{mathpazo}
\usepackage{hyperref}
\usepackage{mathrsfs}
\usepackage{amssymb, amsmath}
\usepackage{enumitem}
\usepackage{xspace}
\usepackage[all]{xy}
\usepackage{bm}

\newcommand{\CC}{\mathbf{C}}
\newcommand{\NN}{\mathbf{N}}
\newcommand{\ZZ}{\mathbf{Z}}
\newcommand{\RR}{\mathbf{R}}
\newcommand{\SSS}{\mathbf{S}}
\newcommand{\QQ}{\mathbf{Q}}

\newcommand{\SL}{\mathbf{SL}}

\newcommand{\sG}{\mathscr{G}}
\newcommand{\sY}{\mathscr{Y}}
\newcommand{\sT}{\mathscr{T}}
\newcommand{\parts}{\mathscr{P}}

\newcommand{\cat}{{\upshape CAT($0$)}\xspace}

\newcommand{\se}{\subseteq}

\newcommand{\fhi}{\varphi}

\newcommand{\inv}{^{-1}}
\newcommand{\action}{\curvearrowright}

\newcommand{\hb}{\mathrm{H}_{\mathrm{b}}}
\newcommand{\hbc}{\mathrm{H}_{\mathrm{cb}}}

\DeclareMathOperator{\Aut}{Aut}
\DeclareMathOperator{\Homeo}{Homeo}
\DeclareMathOperator{\Prob}{Prob}
\DeclareMathOperator{\Ends}{Ends}
\DeclareMathOperator{\Br}{Br}
\DeclareMathOperator{\Bund}{Bund}
\DeclareMathOperator{\Con}{Con}
\DeclareMathOperator{\Clo}{Cl}
\DeclareMathOperator{\Ramen}{Ramen}
\DeclareMathOperator{\Fix}{Fix}

\theoremstyle{plain}
\newtheorem{thm}{Theorem}[section]
\newtheorem*{thm*}{Theorem}
\newtheorem{lem}[thm]{Lemma}
\newtheorem{prop}[thm]{Proposition}
\newtheorem{cor}[thm]{Corollary}
\theoremstyle{definition}
\newtheorem*{defn*}{Definition}
\newtheorem{defn}[thm]{Definition}
\newtheorem{example}[thm]{Example}
\newtheorem*{example*}{Example}
\newtheorem{rem}[thm]{Remark}

\newtheorem{prob}[thm]{Problem}
\newtheorem*{rem*}{Remark}
\begin{document}
\title{Group actions on dendrites and curves}
\author[B. Duchesne]{Bruno Duchesne}
\address{Institut Élie Cartan, Université de Lorraine, Nancy, France.}
\email{bruno.duchesne@univ-lorraine.fr}
\thanks{B.D. is supported in part by French Project ANR-14-CE25-0004 GAMME}
\author[N. Monod]{Nicolas Monod}
\address{EPFL, 1015 Lausanne, Switzerland.}
\begin{abstract}
We establish obstructions for groups to act by homeomorphisms on dendrites. For instance, lattices in higher rank simple Lie groups will always fix a point or a pair. The same holds for irreducible lattices in products of connected groups. Further results include a Tits alternative and a description of the topological dynamics.

We briefly discuss to what extent our results hold for more general topological curves.
\end{abstract}
\maketitle

\section{Introduction}
\textbf{1.A --- Dendrites.}
Recall that a \emph{dendrite} is a locally connected continuum without simple closed curves. Other equivalent definitions and some basic facts are recalled in Section~\ref{sec:prelim} below.

A simple example of a dendrite is obtained by compactifying a countable simplicial tree (Section~\ref{sec:further}), but the typical dendrite is much more complicated: certain Julia sets are dendrites~\cite{Bruin-Todd} and Wa\.zewski's universal dendrite~\cite{WazewskiPHD} can be identified with the Berkovich projective line over $\CC_p$~\cite{Hrushovski-Loeser-Poonen}. The homeomorphisms group of some dendrites is enormous (Section~\ref{sec:further}), whilst for others it is trivial~\cite[p.~443]{deGroot-Wille}.

In fact, it was recognized early on~\cite{Whyburn28,Whyburn30} that \emph{any} continuum has a canonical dendrite quotient reflecting its cut-point structure, see~\cite{Bowditch98}, \cite{Papasoglu-Swenson} for the general statement. Bowditch made remarkable use of this quotient~\cite{BowditchAMS}, leading to the solution of the cut-point conjecture for hyperbolic groups~\cite{Bowditch99,Swarup96}. In this application, Bowditch's dendrites retained a decidedly geometric aspect inherited from the hyperbolic group, namely the dynamical convergence property, allowing him to reconstruct a metric tree with an isometric action.

However, in general, continuous actions on dendrites are definitely not geometrisable (see Section~\ref{sec:further}). Nonetheless, the purpose of our work is to establish rigidity results for actions on dendrites in full generality. It turns out that this context of topological dynamics still admits analogues of some results that are known in geometry of negative curvature. We shall call an action on a dendrite \emph{elementary} if it fixes a point or a pair of points, cf.\ the discussion in Section~\ref{sec:elem}.

\begin{thm}\label{thm:higher}
Let $\Gamma$ be a lattice in a simple algebraic group of rank at least two.

Then any $\Gamma$-action on a dendrite is elementary.
\end{thm}

In the above statement, a \emph{simple algebraic group} refers to $\mathbf{G}(k)$ where $k$ is a local field and $\mathbf{G}$ is a connected almost $k$-simple algebraic group defined over $k$; its \emph{rank} is the $k$-rank of $\mathbf{G}$. Examples include the Lie group $\SL_n(\RR)$, which is of rank $n-1$, or $p$-adic and function field analogues.

\medskip

To highlight one of the differences between general dendrites and trees, we recall that the analogue of Theorem~\ref{thm:higher} for trees or $\RR$-trees is a direct consequence of Kazhdan's property~(T)~\cite{Watatani82}, \cite[6.11]{Harpe-Valette}. In contrast, we are not aware of any possible connection between Kazhdan's property and dendrites.

\begin{prob}\label{prob:T}
Find a Kazhdan group with a non-elementary action on a dendrite.
\end{prob}

Our forthcoming paper~\cite{DM_dendrites_2} on the structural properties of dendrite groups also contains evidence that Kazhdan's property \emph{should not} be an obstruction to actions on dendrites.

\medskip

Our approach is therefore different: we establish a degree two cohomological invariant for actions on dendrites which is a topological version of invariants known in non-positive curvature~\cite{Monod-ShalomCRAS}, \cite{Monod-Shalom1}. In view of the cohomological vanishing results for higher rank lattices proved in~\cite{Burger-Monod1}, \cite{Monod-Shalom1}, the following already accounts for Theorem~\ref{thm:higher}.

\begin{thm}\label{thm:coho-intro}
Let $G$ be a group with a non-elementary action on a dendrite.

There is a canonical unitary representation $V$ of $G$ without invariant vectors and a non-trivial canonical element of the second bounded cohomology $\hb^2(G, V)$.
\end{thm}

This cohomological theorem can be used in other contexts too; here is an example. A lattice  $\Gamma$ in a product $G_1\times \cdots \times G_n$ is called \emph{irreducible} if its projection to any proper sub-product is dense. This definition coincides with the classical one for semi-simple Lie groups and is discussed at length e.g.\ in~\cite[4.A]{CM09}.

\begin{thm}\label{thm:irred}
Let $\Gamma$ be an irreducible lattice in a product of at least two connected locally compact groups.

Then any $\Gamma$-action on a dendrite is elementary.
\end{thm}

In the preliminaries for the proofs of the above theorems, we need to investigate probability measures on dendrites. As a by-product, we obtain a short proof of the fact that every action of an \emph{amenable} group on a dendrite is elementary, first established in~\cite{ShiYe2016}. In fact, this holds in a much wider generality.

\begin{thm}\label{thm:coamen}
If $G$ has two commuting co-amenable subgroups, then any $G$-action on a dendrite is elementary.
\end{thm}

See Section~\ref{sec:coamen} for context; for instance, \emph{any} group whatsoever can be embedded in a group admitting two commuting co-amenable subgroups.

Special cases of this theorem include the following, noting that the case of $F$ answers Problem~6 in~\cite{ShiYe2016}.

\begin{cor}\label{cor:F-etc}
Consider Thompson's group $F$ or any of the groups of piecewise M\"obius transformations of the line introduced in~\cite{Monod_PNAS}.

Then any action of these groups on a dendrite is elementary.
\end{cor}

However, these special cases also follow from a \emph{Tits alternative} that we establish:

\begin{thm}\label{thm:Tits}
Let $G$ be a group acting on a dendrite.

Then either $G$ contains a non-abelian free subgroup or its action is elementary.
\end{thm}

In contrast to some of our other arguments, this is proved rather directly in parallel with the classical cases of trees and $\RR$-trees (cf.\ e.g.~\cite{Pays-Valette}).

The particular case of (topologically) minimal actions was established earlier in~\cite{Shi2012} using ergodic tools; unfortunately, it seems that one cannot reduce general actions on dendrites to that minimal case.

\medskip

The dynamics of an individual homeomorphism of a dendrite can exhibit diverse behaviours; we give a ``tectonic'' description in Proposition~\ref{prop:tectonic} below. However, from the perspective of non-elementary groups, we recover a global phenomenon reminiscent of negative curvature:

\begin{thm}\label{thm:prox-intro}
Let $G$ be a group with a non-elementary action on a dendrite $X$.

Then $X$ contains a canonical compact $G$-set which is a $G$-boundary in Furstenberg's sense.
\end{thm}

\bigskip

\textbf{1.B --- Curves.}
The reader might wonder whether the results presented here really depend on the dendrite structure or just on the one-dimensionality of these topological spaces. For instance, it has been proved that lattices in higher rank algebraic groups have strong obstructions to acting on one-manifolds~\cite{Witte94, Ghys99, Burger-Monod1}. This regards mostly the circle, because much less is known about actions on the interval (which are elementary in our sense anyway).

\medskip
Consider thus a general compact Hausdorff space of dimension one; for instance, a topologist's \emph{curve}, namely a continuum of dimension one. It turns out that our results \emph{emphatically do not hold} in this generality. For instance, any countable residually finite group acts \emph{freely} on the Menger curve~\cite{Menger26}, as follows from~\cite[Thm.~1]{Dranishnikov88} or~\cite{Sakai94}. Such groups include all lattices of Theorem~\ref{thm:higher} above; see Section~\ref{sec:curves} for further discussion.

\medskip
There is however a class of better behaved curves, namely \emph{local dendrites}. By definition, this refers to any continuum in which every point has a neighbourhood which is a dendrite. In fact, a one-dimensional continuum is a local dendrite if and only if it is an absolute neighbourhood retract~\cite{Ward60}, \cite[5.1]{Nadler-PC}. For instance, the Berkovich space of any connected projective scheme of pure dimension one over a separable complete valued field is a local dendrite~\cite[8.1]{Hrushovski-Loeser-Poonen}.

Combining Theorem~\ref{thm:higher} with rigidity results for the circle, we obtain the following.

\begin{cor}\label{cor:higher:loc}
Let $\Gamma$ be a lattice in a simple algebraic group of rank at least two.

Then any $\Gamma$-action on a local dendrite has a finite orbit.
\end{cor}

In Section~\ref{sec:curves}, we further combine our other results for dendrites with known results for the circle and establish:

\begin{itemize}
\item a cohomological obstruction to actions on local dendrites,
\item a result for groups with commuting co-amenable subgroups,
\item a Tits alternative for actions on local dendrites.
\end{itemize}

\tableofcontents

\section{Preliminaries on dendrites}\label{sec:prelim}

\begin{flushright}
\begin{minipage}[t]{0.7\linewidth}\itshape\small
On voit autour des expansions dendritiques une infinit\'e de petits appendices
\begin{flushright}
\upshape\small
Micheline Stefanowska (1897), p.~359 in~\cite{Stefanowska}.
\end{flushright}
\end{minipage}
\end{flushright}
\vspace{3mm}

Recall that a \emph{continuum} is a (non-empty) connected compact metrisable space. A continuum $X$ is a \emph{dendrite} if and only if any one of the following equivalent conditions is satisfied:

\smallskip
\begin{enumerate}[label=(\roman*)]
\item Any two distinct points of $X$ can be separated by a point.\label{pt:sep}
\item $X$ is locally connected and contains no simple closed curve.
\item The intersection of any two connected subsets of $X$ remains connected.\label{pt:intersect}
\item $X$ is a one-dimensional absolute retract.
\item $C(X)$ is projective in the category of unital C*-algebras.
\end{enumerate}

\noindent
See e.g.~\cite{Chacha}, \cite{Chigogidze-Dranishnikov}; a reference containing all those preliminary facts that we do not justify is~\cite[\S10]{Nadler}.

\medskip

Any non-empty closed connected subset of a dendrite is path-connected and is again a dendrite. The characterisation~\ref{pt:intersect} implies that any subset $A$ of a dendrite is contained in a unique minimal closed connected subset, which we denote by $[A]$. (A closed connected subset is automatically a \emph{sub-dendrite}, i.e.\ itself a dendrite.) When $A$ consists of two points $x,y$ we denote this dendrite simply by $[x,y]$; it is an \emph{arc}, i.e.\ a homeomorphic image of a compact interval in the real line. We define the \emph{interior} of an arc $[x,y]$ to be $[x,y]\setminus\{x,y\}$. Yet another characterisation of dendrites amongst continua is:

\smallskip
\begin{enumerate}[label=(\roman*)]\setcounter{enumi}{5}
\item Any two points are the extremities of a unique arc in $X$.
\end{enumerate}

\smallskip

For any point $x$ in a dendrite $X$, all connected components of $X\setminus \{x\}$ are open (by local connectedness). The cardinality of this set of component coincides with the Menger--Urysohn \emph{order} of $x$ in $X$~\cite[\S46, I]{Kuratowski_T2}. The point $x$ is called an \emph{end point} if it has order~$1$ and a \emph{branch point} if it has order~$\geq 3$. We denote by $\Ends(X)$ and $\Br(X)$ the subsets of end points and branch points. The other points, of order~$2$, are called \emph{regular}. There are only countably many branch points (throughout this article, \emph{countable} means~$\leq \aleph_0$). The set of end points is non-empty and the set of regular points is dense; in fact, dense in any arc.

\bigskip

The following is a simple topological version of Helly's theorem.

\begin{lem}\label{Helly}
Let $X$ be a dendrite and $\sY$ be a collection of closed connected subsets $Y\se X$ that have pairwise non-empty intersection. Then the intersection of all members of $\sY$ is non-empty.
\end{lem}

\begin{proof}
By compactness, we can assume that $\sY$ is finite. Choose a point $x_{Y,Y'}$ in each $Y\cap Y'$. We can replace $X$ by the compact tree $X_0=[\{x_{Y,Y'}: Y, Y'\in \sY\}]$ and each $Y$ by $Y\cap X_0$, which is connected. Now the result follows from Helly's theorem for trees (where connectedness coincides with convexity).
\end{proof}

We will also need the following form of acyclicity.

\begin{lem}\label{lem:carre}
Let $z_i$ be closed connected subsets of a dendrite $X$ indexed by $i\in \ZZ/4\ZZ$. If $z_i\cap z_{i+1}\neq\varnothing$ for all $i$, then either $z_0\cap z_1\cap z_2\neq\varnothing$ or $z_1 \cap z_2 \cap z_3\neq\varnothing$.
\end{lem}

\begin{proof}
Apply Lemma~\ref{Helly} to the three sets $z_0\cup z_3$, $z_1$, $z_2$.
\end{proof}

End points can be separated from connected sets in the following sense.

\begin{lem}\label{lem:KM}
Let $C$ be a closed subset of a dendrite $X$ and $x\in \Ends(X)$. If $x\in [C]$, then $x\in C$.
\end{lem}

\begin{proof}
It follows from the definition of the Menger--Urysohn order that $x$ admits a system of open neighbourhoods whose topological boundary is reduced to a point (see also~\cite[9.3]{Nadler}). Since $X$ is a dendrite, the complement of any such neighbourhood is connected, and the result follows.
\end{proof}

We endow as usual the space $\Clo(X)$ of non-empty closed subsets of $X$ with the Vietoris topology, which is compact and metrisable~\cite[\S1]{Illanes-Nadler}.

\begin{lem}\label{subcont}
The map $C\mapsto[C]$ is continuous in $\Clo(X)$.
\end{lem}

\begin{proof}
Recall that the Vietoris topology can be metrised using the Hausdorff distance associated to a compatible distance on $X$. One then checks directly that the above map is continuous on finite subsets $C$ by considering the tree spanned by the union of two given finite subsets of $X$. In order to deduce the general case, one uses the uniform approximation of $X$ by trees~\cite[10.37]{Nadler}.
\end{proof}

Finally, we recall that dendrites have the fixed-point property:

\begin{lem}\label{lem:FP}
Every homeomorphism of a dendrite fixes a point.
\end{lem}

\begin{proof}
The first occurrence of this statement is probably~\cite[V, p.~129]{Scherrer25}.

Since $X$ is an absolute retract~\cite[\S48 III 16]{Kuratowski_T2}, this can also be viewed as a consequence of Schauder's theorem~\cite{Schauder30}, or even of a fixed-point result predating Brouwer's, such as~\cite[p.~186]{Bohl04}, because $X$ can be embedded in the plane~\cite[p. ~9]{WazewskiPHD}.
\end{proof}

As the above references show, this holds more generally for continuous self-maps. We recall however that there are \emph{tree-like} continua for which the fixed-point property fails for continuous self-maps~\cite{Bellamy80} and even for homeomorphisms~\cite{Fugate-Mohler}, \cite{Oversteegen-Rogers}.

\section{Elementarity of actions on dendrites}\label{sec:elem}
All group actions on dendrites will be understood to be by homeomorphisms. If a group $G$ has a topology, the action on $X$ is called continuous when the action map $G\times X \to X$ is so.

\begin{defn}
An action of a group $G$ on a dendrite $X$ is \emph{elementary} if $G$ fixes a point in $X$ or preserves a pair of points in $X$.
\end{defn}

There is not much to say about elementary actions; any countable group admits an action on a dendrite which is free away from a single fixed point.

This is a notion of triviality for actions that could have been defined in several other equivalent ways:

\begin{prop}\label{prop:elementary}
Let $G$ be a group acting on a dendrite $X$. The following are equivalent.
\begin{enumerate}[label=(\roman*)]
\item The $G$-action is elementary.\label{pt:el:el}
\item $G$ preserves an arc in $X$, possibly reduced to a point.
\item $G$ has a finite orbit in $X$.\label{pt:el:fin}
\item There is a $G$-invariant probability measure on $X$.\label{pt:el:proba}
\end{enumerate}
\end{prop}

An immediate consequence of this proposition is the fact that every continuous action of an \emph{amenable} group on a dendrite is elementary, first proved in~\cite{ShiYe2016}.

\smallskip

The only non-trivial implication needed for the above equivalences, namely~\ref{pt:el:proba}$\Rightarrow$\ref{pt:el:el}, is a particular case of the following result, which will however be particularly useful for \emph{non-elementary} actions.

\begin{prop}\label{prob}
Let $X$ be a dendrite. There is a Borel $\Homeo(X)$-equivariant map
$$\fhi\colon\Prob(X)\longrightarrow\parts_{1,2}(X)$$
to the space $\parts_{1,2}(X)$ of subsets of cardinality $1$ or $2$ in $X$.
\end{prop}

To be precise, we view $\parts_{1,2}(X)$ as a compact subset of the hyperspace $\Clo(X)$; this is none other than the Borsuk--Ulam symmetric product~\cite{Borsuk-Ulam} of two copies of $X$. As for $\Prob(X)$, it denotes the compact convex space of Borel probability measures on $X$ endowed with the usual weak-* topology (in the dual of the space of continuous functions).

\begin{proof}
Given a probability measure $\mu$ on $X$ we construct below an element of $\parts_{1,2}(X)$. The resulting map $\fhi$ will be canonical enough to be $\Homeo(X)$-equivariant by definition, and it will be explicit enough to be Borel by a straightforward verification (using also Lemma~\ref{subcont}). Therefore we only provide the construction.

Suppose first that $\mu$ has atoms. It then has finitely many atoms of maximal mass. It is therefore enough to construct a map $\parts_{\mathrm{f}}(X)\to \parts_{1,2}(X)$ on the space $\parts_{\mathrm{f}}(X) \se\Clo(X)$ of finite non-empty subsets. For any $A\in \parts_{\mathrm{f}}(X)$, the set $[A]$ is a tree; we work with trees without degree two vertices to remain well-defined topologically. There is a number of different well-know canonical ways to associate a notion of center to a finite tree. We choose Jordan's center~\cite{Jordan1869} because it is the most classical that we know of; it is indeed a set of one or two vertices of $[A]$.

\smallskip
We now consider the atom-free case and distinguish two subcases. Assume first that there exist regular points $x$ such that both components of $X\setminus \{x\}$ have measure~$1/2$. We claim that the set $E_\mu \se X$ of all such points $x$ lies in an arc; then, the extremities of $[E_\mu]$ yield the desired element of $\parts_{1,2}(X)$. If the claim did not hold true, there would be elements $x,y,z\in E_\mu$ forming a tripod, yielding three disjoint components of mass~$1/2$, which is impossible.

In the second subcase, every regular point $x\in X$ determines a component $c_x$ of $X\setminus \{x\}$ with mass~$>1/2$, recalling that $\mu(\{x\})=0$. We claim that the intersection $\bigcap  \overline{c_x}$, where $x$ ranges over all regular points of $X$, is a singleton. This will be our element of $\parts_{1,2}(X)$. To prove this claim, we observe that the closed connected sets $\overline{c_x}$ have pairwise non-empty intersection by virtue of their mass. Thus, by Lemma~\ref{Helly}, the intersection $\bigcap  \overline{c_x}$ is non-empty. Since any two distinct points are separated by a regular point, we proved the claim.
 \end{proof}

\section{Dendro-minimality}
Recall that any action on any compact space admits (possibly several) minimal invariant non-empty closed subsets. A similar application of Zorn's lemma shows that an action on a dendrite admits \emph{some} minimal invariant sub-dendrite. In fact, more is true:

\begin{lem}\label{lem:min}
Let $G$ be a group with a non-elementary action on a dendrite $X$.
\begin{enumerate}[label=(\roman*)]
\item There is a unique minimal $G$-invariant non-empty closed subset $M\se X$.\label{pt:min:top}
\item There is a unique minimal $G$-invariant sub-dendrite in $X$, namely $[M]$.\label{pt:min:den}
\end{enumerate}
\end{lem}

\begin{proof}
The first point is proved in~\cite[4.1]{Marzougui-Naghmouchi} and~\ref{pt:min:den} follows from~\ref{pt:min:top}.
\end{proof}

\begin{rem}
In fact, there is a unique minimal $G$-invariant sub-dendrite in $X$ as soon as $G$ has no global fixed points. Indeed, minimal $G$-invariant sub-dendrites are necessarily disjoint, and the \emph{first-point retraction} (see 10.24 and~10.25 in~\cite{Nadler}) would map any invariant sub-dendrite to a fixed point in any disjoint invariant sub-dendrite. Thus we see that there is still a canonical minimal $G$-set when there are no $G$-fixed points: in the elementary case, take the extremities of the unique minimal invariant arc.

There are however such examples without uniqueness of minimal sets. The simplest is provided by the following action of the infinite dihedral group $D\cong \ZZ\rtimes \{\pm 1\}$ on the H-shaped dendrite. The $D$-action on the interior of the horizontal rung of H is conjugated to its standard isometric action on the real line, while the ends of H are permuted according to the appropriate quotient map from $D$ to the Vierergruppe~\cite[pp.~12--13]{Klein84}.
\end{rem}

Classically, an action on a space $M$ without invariant closed proper subsets is called (topologically) \emph{minimal}. There is of course no reason that $M$ should be a dendrite; therefore, we shall need to consider rather the case where the action is \emph{dendro-minimal}, that is, does not admit an invariant proper sub-dendrite.

\begin{lem}\label{lem:dis}
Let $G$ be a group with a dendro-minimal action on a dendrite $X$. For any proper sub-dendrite $Y\se X$, one can find $g\in G$ such that $gY\cap Y=\varnothing$.
\end{lem}

\begin{proof}
Otherwise it follows that $gY\cap g'Y \neq \varnothing$ for all $g,g'\in G$. By Lemma~\ref{Helly}, the intersection of all $G$-translates of $Y$ is non-empty; this intersection is a proper $G$-invariant sub-dendrite, which is absurd.
\end{proof}

\begin{lem}\label{lem:dense}
Let $G$ be a group with a dendro-minimal action on a dendrite $X$. Then the closure of any $G$-orbit contains $\Ends(X)$.
\end{lem}

In particular, in the non-elementary dendro-minimal case, the unique minimal set contains $\Ends(X)$.

\begin{proof}[Proof of the lemma]
Let $x\in X$ and $C=\overline{Gx}$. By dendro-minimality, $[C]=X$. It suffices now to apply Lemma~\ref{lem:KM}.
\end{proof}

Thanks to the fixed-point property of dendrites, Lemma~\ref{lem:min} has the following consequence.

\begin{lem}\label{lem:com}
Let $G$ be a group with a non-elementary action on a dendrite $X$ and let $M$ be the unique minimal subset of Lemma~\ref{lem:min}.

Then any homeomorphism of $X$ commuting with $G$ fixes $M$ point-wise.
\end{lem}

\begin{proof}
Let $h$ be a homeomorphism of $X$ commuting with $G$. Then $G$ preserves the closed subset of $h$-fixed points, which is non-empty by Lemma~\ref{lem:FP}. Therefore, this set contains $M$.
\end{proof}

Lemma~\ref{lem:com} immediately implies the following.

\begin{cor}\label{cor:product}
When a direct product of two groups acts on a dendrite, at least one of the factors acts elementarily.\qed
\end{cor}

\begin{rem}
There is a special case where (topological) minimality and dendro-minimality do coincide. A \emph{free arc} is an arc (not reduced to a point) whose interior is open in $X$; for dendrites, this is equivalent to asking that this interior does not meet $\Br(X)$. For any dendrite $X$ not reduced to a point, the following conditions are equivalent:

\smallskip
\begin{enumerate}[label=(\roman*)]
\item $X$ has no free arc.\label{pt:free:arc}
\item $\Br(X)$ is dense in $X$.\label{pt:Br:dense}
\item $\Ends(X)$ is dense in $X$.\label{pt:Ends:dense}
\end{enumerate}

\noindent
(The implications \ref{pt:free:arc}$\Rightarrow$\ref{pt:Br:dense} and \ref{pt:Ends:dense}$\Rightarrow$\ref{pt:free:arc} follow readily from the definitions and \ref{pt:Br:dense}$\Rightarrow$\ref{pt:Ends:dense} can be found e.g.\ in~\cite[Prop.~2.3]{Charatonik91}.)

\smallskip
Now Lemma~\ref{lem:dense} implies that minimality and dendro-minimality are equivalent for any action on a dendrite satisfying the above three equivalent conditions.
\end{rem}

\section{A Tits alternative for dendrites}\label{sec:Tits}
We will apply the following well-known variant of Klein's Ineinanderverschiebungsprozess (cf.\ III \S16,18 in~\cite{Klein83}), also described by the less precise, but less sesquipedalian, terminology ``ping-pong lemma'':

Let $a,b$ be two elements of a group $G$ acting on a set $X$ and suppose that $X$ contains four non-empty disjoint sets $A_\pm, B_\pm$ such that
$$a^{\pm 1} (X\setminus A_\mp) \se A_\pm, \kern 5mm b^{\pm 1} (X\setminus B_\mp) \se B_\pm.$$
Then $a,b$ are free generators of a free subgroup of $G$. (This criterion can be easily deduced from the statements given e.g.\ in~\cite[II.24]{Harpe_GT} or~\cite{MacBeat63}.)

\medskip
\begin{proof}[Proof of Theorem~\ref{thm:Tits}]
It suffices to consider the case where the $G$-action on a dendrite $X$ is non-elementary and dendro-minimal. In particular, we can choose two points $x,y\in X$ such that the arc $[x,y]$ contains at least some branch point $r\neq x,y$. We further choose regular points $p\in [x,r]$ and $q\in [r,y]$ with $p\neq x$ and $q\neq y$.

We use the following notation: given two distinct points $s,t\in X$, we write $U_s(t)$ for the sub-dendrite defined as the closure in $X$ of the component of $X\setminus\{s\}$ containing $t$.

Since $p$ is regular, $U_p(x)$ and $U_p(y)$ cover $X$. Thus, by Lemma~\ref{lem:dis}, there is $g\in G$ such that $g (U_p(y))\se U_p(x)$. Likewise, there is $h\in G$ such that $h (U_q(x))\se U_q(y)$. Using $U_p(x)\se U_q(x)$, we deduce
$$hg (U_p(y))\se  U_q(y).$$
If we set $a=hg$, $A_-=U_p(x)$ and $A_+ = a(X\setminus A_-)$, then the two conditions $a^{\pm 1} (X\setminus A_\mp ) = A_\pm$ hold by definition. Moreover, $A_-$ and $A_+$ are disjoint since $A_+ \se a(U_p(y))$ which lies in $U_q(y)$ by the above discussion.

We now intend to define $b=f a f\inv$ as a conjugate of $a$ by an appropriate $f \in G$, so that the corresponding properties will automatically hold for $B_\pm = f(A_\pm)$. The only requirement to secure is that the resulting two sets $B_\pm$ are disjoint from both $A_\pm$.

To this end, denote by $Y\se X$ the sub-dendrite obtained as the union of $U_p(x)$, $U_q(y)$ and $[p,q]$. By construction, $Y$ contains $A_\pm$. Since we have made sure that $[p,q]$ contains some branch point $r$, the sub-dendrite $Y$ is not all of $X$. Therefore, by Lemma~\ref{lem:dis}, there is $f\in G$ that  $f(Y)\cap Y=\varnothing$. This completes the proof.
\end{proof}

\begin{cor}
Any torsion group acting on a dendrite has a fixed point.
\end{cor}

\begin{proof}
It is well-known (and obvious) that the group of orientation-preserving homeomorphisms of an arc has no element of finite order. Therefore, if $G$ is torsion, any $G$-action on an arc factors through a group of order at most two and hence fixes a point. The result now follows since Theorem~\ref{thm:Tits} shows that $G$ can only act elementarily on a dendrite.
\end{proof}

\section{Co-amenable subgroups}\label{sec:coamen}
Let $G$ be a topological group. A subgroup $H$ of $G$ is \emph{co-amenable} if any continuous affine $G$-action on a convex compact set (in a Hausdorff locally convex topological vector space) has a fixed point whenever it has an $H$-fixed point. See~\cite{Eymard72} for equivalent definitions in the locally compact case. A basic example is when $H$ is a lattice in $G$.

\smallskip
For \emph{normal} subgroups, co-amenability is simply the amenability of the quotient; in general, the situation is much more interesting. For instance, consider the following property of a group $G$: \itshape The group $G$ admits two commuting co-amenable subgroups\upshape.

\smallskip
It was proved in~\cite[\S 2.A]{CM09} that this condition implies all the known consequences of amenability in the \cat setting (in particular it reduces to amenability in the linear case). However, they are many non-amenable groups with this property. For instance, for any group $Q$, the wreath product $G=Q\wr \ZZ$ enjoys this property. This is because the two subgroups $\bigoplus_{\ZZ>0} Q$ and  $\bigoplus_{\ZZ<0} Q$ are co-amenable~\cite{Monod-Popa}. The non-amenable groups of piecewise projective homeomorphisms of the line introduced in~\cite{Monod_PNAS} also have this property, as does the undecided Thompson group $F$ (see e.g.~\cite[\S 2.A]{CM09}).

\medskip

The proof of Theorem~\ref{thm:coamen} hinges on the following fact.

\begin{lem}\label{coamen}
Let $G$ be a group with a continuous action on a dendrite and with a co-amenable subgroup $H$. Then the $H$-action is elementary if and only if the $G$-action is so.
\end{lem}

\begin{proof}[Proof of the lemma]
We use Proposition~\ref{prop:elementary}. If $H$ acts elementarily, then there is an $H$-invariant probability measure on $X$. By co-amenability, there is also a $G$-invariant probability measure. Therefore, the $G$-action is elementary. The converse is immediate.
\end{proof}

\begin{proof}[Proof of Theorem~\ref{thm:coamen}]
Let $H_1,H_2\leq G$ be two commuting co-amenable subgroups. Any $G$-action on a dendrite gives rise to a $H_1\times H_2$-action. By Corollary~\ref{cor:product}, one of the factors acts elementarily. Now Lemma~\ref{coamen} shows that the $G$-action is elementary.
\end{proof}

\section{The fundamental bundle}
We now proceed to define the \emph{fundamental bundle} of a dendrite $X$, which will be a locally compact second countable space $\Bund(X)$ endowed with a topological quotient map $\Bund(X)\to X$. The points of $\Bund(X)$ are pairs $(x,c)$ with $x\in X$ and $c$ a component of $X\setminus\{x\}$; we dot not define $\Bund(X)$ when $X$ is reduced to a point.

We recall that all components of $X\setminus\{x\}$ are open in $X$; in particular there are countably many for each $x$. The fibre above $x$ will naturally identify with the discrete space $\pi_0(X\setminus\{x\})$.

\medskip

In order to topologize $\Bund(X)$, we map it injectively into the product space $X \times \Con(X)$, where $\Con(X)\se\Clo(X)$ denotes the hyperspace of closed connected subsets of $X$, i.e.\ of sub-dendrites, endowed with the Vietoris topology, which is metrisable and compact~\cite[3.7]{Illanes-Nadler}. The map is given by
$$(x, c) \longmapsto (x, \{x\} \cup c)$$
and thus the projection on the first coordinate is indeed a quotient map $\Bund(X)\to X$ for the corresponding topology. The only point left to justify is that $\Bund(X)$ is locally compact, or equivalently~\cite[I.20, I.66]{BourbakiTGI_alt} that its image can be written as the intersection of an open and a closed subset in $X \times \Con(X)$.

We claim that this image is closed in the set of pairs $(x,z)$ with $z$ not a singleton (which is an open set~\cite[1.15]{Illanes-Nadler}). Suppose indeed that a sequence $(x_n, \{x_n\} \cup c_n)$ as above converges to $(x,z)$. Then $x\in z$ holds; suppose for a contradiction that $z\setminus \{x\}$ is not a component of $X\setminus \{x\}$. It is, however, connected and non-empty by hypothesis; let thus $c$ be its component and $p\in c\setminus z$. Then $[p,x]$ meets $z$ at a point $q\neq x$. If $U$ is a small enough connected open neighbourhood of $q$, then $x_n$ is outside $\overline{U}$ and outside $[p,q]$ for $n$ large enough because $x_n\to x$. Therefore, $U$ and $p$ lie in the same component of $X\setminus \{x_n\}$. However, the definition of the Vietoris topology for $\{x_n\} \cup c_n\to z$ implies that for large enough $n$ each $\{x_n\} \cup c_n$, hence also each component $c_n$, avoids $p$ but meets the open set $U$ non-trivially (because $z$ does so). This is a contradiction.

\medskip

For later use, we also define the \emph{double} fundamental bundle $\Bund^2(X)$ to be the fibred product of two copies of $\Bund(X)$ over $X$. Thus $\Bund^2(X)$ is a topological bundle with discrete countable fibres over $X$ and its points are pairs of components corresponding to a same point.

\medskip

Since these constructions are natural, the group of homeomorphisms of $X$ acts by homeomorphisms on $\Bund(X)$ and on $\Bund^2(X)$; moreover, the maps to $X$ as well as the projections $\Bund^2(X)\to \Bund(X)$ are equivariant for these actions.

\section{Furstenberg maps}\label{sec:Furstenberg}
Let $G$ be a locally compact second countable group and $B$ a $G$-measure space, which means a standard measure space with a measurable $G$-action preserving the measure class. We shall be particularly interested in the case where $B$ is the Poisson--Furstenberg boundary of a (spread-out, generating) random walk on $G$. Slightly varying definitions~\cite{Burger-MonodGAFA,Bader-Furman_announce} have been introduced for a general $G$-measure space that satisfies the most desirable properties of this particular example. We choose the strongest definition~\cite{Bader-Furman_announce}:

\begin{defn}\label{defn:bd}
A \emph{strong $G$-boundary} is a $G$-measure space such that
\begin{itemize}
\item the action $G\action B$ is amenable in Zimmer's sense~\cite{Zimmer84},
\item the projections $B\times B\to B$ are isometrically ergodic~\cite{Bader-FurmanICM}.
\end{itemize}
\end{defn}

We refer to~\cite{Bader-Furman_announce}, to~\cite{Bader-FurmanICM} and to~\cite[\S4]{BDL16} for more details and for a proof that Poisson--Furstenberg boundaries are strong boundaries. We shall only need the fact that some strong boundary exists for each $G$.

Regarding the second condition in Definition~\ref{defn:bd}, we will use only the following two particular cases of it:

1)~\emph{Ergodicity with coefficients} of the diagonal $G$-action on $B\times B$, which means that any measurable equivariant map from $B\times B$ to any separable dual Banach $G$-module is essentially constant~\cite{Burger-MonodGAFA}. This implies in particular the usual ergodicity of the action of $G$ or of any lattice in $G$ on $B\times B$ and hence also on $B$.

2)~A lifting property: suppose that $Z$ and $X$ are standard Borel $G$-spaces with a Borel $G$-map $Z\to X$ with countable fibres. If there are measurable $G$-maps $f\colon B\times B\to Z$ and $B\to X$ such that the diagram
$$\xymatrix{
B\times B  \ar[d]\ar[rr]^f && Z \ar[d]\\
B  \ar[rr]\ar@{.>}[rru] && X
}$$
commutes a.e.\ for the first projection $B\times B\to B$, then there is a measurable $G$-map $B\to Z$ such that the diagram commutes a.e.; in other words, $f$ is essentially independent of the second variable. In the terminology of~\cite{Bader-FurmanICM} this application of the definition consists simply in viewing the fibres as separable metric spaces for the discrete metric, which constitutes indeed a Borel field of metric spaces.

\begin{thm}\label{bmap}
Let $G$ be a  locally compact second countable group with a continuous non-elementary action on a dendrite $X$. If $B$ is a strong $G$-boundary, then there exists a measurable $G$-map $B\to X$.

If moreover the $G$-action on $X$ is dendro-minimal, then the map $B\to X$ ranges in $\Ends(X)$ and is unique (up to null-sets).
\end{thm}

\begin{rem}\label{rem:bmap}
In the non-dendro-minimal case, the above statement still provides a \emph{canonical} map, namely the unique map to the ends of the unique minimal sub-dendrite of Lemma~\ref{lem:min}.
\end{rem}

We start with a transversality lemma; below, $\parts_2(X)\se\Clo(X)$ denotes the space of sets of two points in $X$.

\begin{lem}\label{lem:pairs}
Under the initial assumptions of Theorem~\ref{bmap}, any measurable $G$-map $\fhi\colon B\to \parts_2(X)$ satisfies $[\fhi(b)]\cap[\fhi(b')]=\varnothing$ for almost all $(b,b')\in B^2$.
\end{lem}

\begin{proof}[Proof of the lemma]
Otherwise, by double ergodicity, $[\fhi(b)]\cap[\fhi(b')]$ is non-empty for almost all $(b, b')$ in $B\times B$. In general, a co-null set in $B\times B$ need not contain any product of co-null sets in $B$. However, in our situation, precisely this does happen: we claim that there is a co-null set $A\se B$ such that $[\fhi(b)]\cap[\fhi(b')]$ is non-empty for all $b, b'\in A$.

This claim will contradict the non-elementarity. Indeed, we can assume that $A$ is invariant under a countable dense subgroup $\Lambda$ of $G$. Then the collection $\{[\fhi(b)]\}_{b\in A}$ has non-empty intersection by Lemma~\ref{Helly} and is $\Lambda$-invariant (we can assume that $\fhi$ is strictly $G$-equivariant by~\cite[B.5]{Zimmer84}). This intersection is an arc or a point and is preserved by $G$ by continuity of the action, showing that the $G$-action is elementary.

We now prove the claim. By Fubini, there is a co-null set $B_0\se B$ and for each $b\in B_0$ a co-null set $B_b\se B$ such that $[\fhi(b)]\cap[\fhi(b')]$ is non-empty for all $b'\in B_b$. If $A=B_0$ satisfies the claim, we are done. Otherwise, we can fix $b,c\in B_0$ with $[\fhi(b)]\cap[\fhi(c )] = \varnothing$ and define $A=B_b\cap B_c$. We need to show that for all $b', c'\in B_b\cap B_c$ the intersection $[\fhi(b')]\cap[\fhi(c')]$ is non-empty. This follows by applying Lemma~\ref{lem:carre} to the arcs $[\fhi(b)]$, $[\fhi(b')]$, $[\fhi(c)]$, $[\fhi(c')]$.
\end{proof}

\begin{proof}[Proof of Theorem~\ref{bmap}] 
By Lemma~\ref{lem:min}, we can assume that the action is dendro-minimal. By amenability, there is a measurable $G$-map $B\to\Prob(X)$. We compose it with the map obtained in Proposition~\ref{prob} so that by ergodicity we have either a map $\fhi\colon B\to \parts_{2}(X)$ or a map $\fhi\colon B\to X$.

We claim that in the second case $\fhi$ ranges in $\Ends(X)$. By ergodicity and since $G$ does not fix a point in $X$, the points $\fhi(b)$ and $\fhi(b')$ are distinct for almost every pair $(b,b')$. Therefore we can define a $G$-map $f\colon B\times B\to \Bund(X)$ by $f(b,b')=(\fhi(b), c)$ where $c$ is the component of $\fhi(b')$ in $X\setminus \{\fhi(b)\}$. The lifting property for the bundle $\Bund(X)\to X$ implies that there is a measurable $G$-map $b\mapsto c_b$ such that for almost all $(b, b')$ the point $\fhi(b')$ lies in the component $c_b$ independently of $b'$. In other words, $\fhi$ ranges essentially in $c_b$; dendro-minimality implies readily that $b$ is an end.

We now proceed to rule out the first case. By Lemma~\ref{lem:pairs}, we have $[\fhi(b)]\cap[\fhi(b')]=\varnothing$ almost surely. For such pairs $(b, b')$ there is a unique $(x,c)$ in $\Bund(X)$ such that $x\in [\fhi(b)]$, $c\supseteq \fhi(b')$ and $c\cap \fhi(b)=\varnothing$. We define $f(b,b')=(x,c)$. Applying again the lifting property, we conclude that $(x,c)$ depends on $b$ only. This implies as above that $x$ is almost surely an end; this, however, is impossible since $c\cap \fhi(b)=\varnothing$.

It remains to show the uniqueness of the map $B\to \Ends(X)$. This follows, using ergodicity, from the fact that there is no $G$-map to $\parts_{2}(X)$.
\end{proof}

\begin{cor}\label{cor:ramen}
Let $G$ be a locally compact second countable group with a continuous non-elementary action on a dendrite $X$.

Then the amenable radical $\Ramen(G)$ acts trivially on the unique minimal $G$-invariant set $M$ of Lemma~\ref{lem:min}.
\end{cor}

In general, the conclusion does not hold for all of $X$ instead of $M$. This is illustrated by the following example which is non-elementary, dendro-minimal and even \emph{topologically transitive}.

\begin{example}
Let $Q$ be a group acting on a locally finite simplicial tree $T$, transitively on the set $E$ of unoriented edges; for instance, $\SL_2(\ZZ)$ or $\SL_2(\QQ_p)$. Choose a discrete abelian group $A$ acting minimally on $\RR$, for instance $\ZZ^2$ viewed as $\ZZ[\sqrt 2]$. Let $X$ be the dendrite obtained as end-compactification of the geometric realization of $T$ (thus it is homeomorphic to Gehman's dendrite~\cite{Gehman} in both above cases, if $p=2$, although the trees are different). Then the permutational wreath product
$$G = \Big(\bigoplus_{E} A\Big) \rtimes Q$$
has an action on $X$ obtained by identifying the interior of each edge of $T$ with $\RR$. The action has the claimed properties, but the normal amenable subgroup $\bigoplus_{E} A$ acts faithfully.
\end{example}

The above example is a special case of a situation discussed again after Proposition~\ref{prop:vegetale} below.

\begin{proof}[Proof of Corollary~\ref{cor:ramen}]
Since the action is non-elementary, we can assume that it is dendro-minimal upon replacing $X$ by its minimal $G$-invariant dendrite. Any strong boundary $B$ for $G/\Ramen(G)$ is also a strong boundary for $G$. The existence of a map as in Theorem~\ref{bmap} implies that some points of $X$ are fixed by $\Ramen(G)$. However, the set of $\Ramen(G)$-fixed points is closed and $G$-invariant. Therefore, this set contains $M$.
\end{proof}

\section{An invariant in bounded cohomology}
The first goal of this section is to define a canonical $2$-cocycle $\omega$ for any dendrite $X$ not reduced to a point. This cocycle will be a topological generalisation of the cocycle introduced for trees in~\cite{Monod-ShalomCRAS}; see however Remark~\ref{rem:quasi} below for an important difference.

Given $p,q\in X$ we define a Borel function $\alpha(p,q)$ on the double bundle $\Bund^2(X)$ by
\begin{equation*}
\alpha(p,q)(x,c,c') =
\begin{cases}
1 &\mbox{if } p\in c, q\in c' \mbox{ and } c\neq c'\\
-1 & \mbox{if } p\in c', q\in c \mbox{ and } c\neq c'\\
0 &\mbox{otherwise}
\end{cases}
\end{equation*}
recalling that $c$ and $c'$ are components of $X\setminus \{x\}$ (in particular the cases above are indeed mutually exclusive). Observe that the above expression is alternating in $(p,q)$ and in $(c,c')$, is invariant under the homeomorphisms of $X$ and is non-zero in $x$ if and only if $x\in [p,q]$ with $x\neq p,q$.

Now the cocycle $\omega$ is defined as the homogeneous coboundary of $\alpha$, which in view of the alternation can be written as
\begin{equation*}
\omega(p,q,r) = \alpha(p,q) + \alpha (q,r) + \alpha (r,p).
\end{equation*}
It follows that $\omega$ is a canonical alternating $2$-cocycle with values in alternating Borel functions on $\Bund^2(X)$. By construction, $\omega(p,q,r)$ takes only the values $\pm 1,0$ and indeed vanishes at all points $(x,c,c')$ unless $x$ is the unique point in the intersection of the three arcs $[p,q]$, $[q,r]$ and $[r,p]$. When $x$ is this point, a direct inspection shows that there are at most six pairs $(c,c')$ such that $\omega(p,q,r)(x,c,c')$ is non-zero. More precisely, there are exactly six such pairs when $p,q,r$ span a tripod with center $x$, two pairs when they span an arc but are pairwise disjoint, and none otherwise.

\medskip

It will be convenient to restrict $\omega$ to range in a smaller space. To this end, we denote by $\Lambda(X)$ the sub-bundle of $\Bund^2(X)$ above the branch points of $X$ and consider $\Lambda(X)$ as a set without topology but with its natural action by $\Homeo(X)$. In particular, $\Lambda(X)$ is countable and $\ell^p(\Lambda(X))$ is a separable isometric dual Banach $\Homeo(X)$-module for all $1\leq p < \infty$. We can summarize some of the above discussion as follows.

\begin{prop}\label{prop:omega}
There is a canonical $\Homeo(X)$-equivariant alternating bounded (norm-)Borel cocycle
$$\omega\colon X^3 \longrightarrow \ell^p(\Lambda(X))$$
which is non-zero on all triples in  $X^3$ that are not contained in a common arc.\qed
\end{prop}

One aspect that will be important later on is the following. Although $\omega$ is the coboundary of $\alpha$, which can be restricted to range in  $\ell^\infty(\Lambda(X))$, there is in general no equivariant bounded map ranging in $\ell^p(\Lambda(X))$ with $p<\infty$ of which $\omega$ is the coboundary. This fact will be a by-product of boundary theory.

\begin{rem}\label{rem:quasi}
In the geometric setting of trees~\cite{Monod-ShalomCRAS}, the cocycle could be considered as a \emph{quasification} of the well-known Haagerup $1$-cocycle which underlies the connection between actions on trees and on Hilbert spaces. For dendrites, in contrast, there is no obvious connection to actions on Hilbert spaces; compare Problem~\ref{prob:T}. We can still define a related $1$-cocycle ranging in the space of functions of the (non-double) bundle $\Bund(X)$, but it is unclear how to make any use of it since the space of points (or of branch points) of a given arc admits no invariant measure or mean unless we are in a discrete or non-nesting case --- which would precisely be accessible to tree or $\RR$-tree methods.
\end{rem}

We now obtain a cohomological obstruction to non-elementary group actions on dendrites:

\begin{thm}\label{thm:coho}
Let $G$ be a locally compact second countable group with a non-elementary continuous action on a dendrite $X$. Then $\hbc^2(G, \ell^p(\Lambda(X)))$ contains a canonical non-trivial element for all $1\leq p < \infty$.
\end{thm}

Theorem~\ref{thm:coho-intro} from the introduction follows by choosing $p=2$; we only need to justify that $\ell^2(\Lambda(X))$ contains no (non-zero) $G$-invariant vector. Any non-trivial level-set of such a vector would be a non-empty finite $G$-invariant subset of $\Lambda(X)$ and its projection to $X$ would witness the elementarity of $G$ according to Proposition~\ref{prop:elementary}\ref{pt:el:fin}.

\begin{proof}[Proof of Theorem~\ref{thm:coho}]
Let $B$ be a strong boundary for $G$ as in Section~\ref{sec:Furstenberg}. By Theorem~\ref{bmap}, there exists a measurable $G$-map $\fhi\colon B\to X$; we choose the canonical map of Remark~\ref{rem:bmap}. Combining $\fhi$ with the cocycle $\omega$ of Proposition~\ref{prop:omega}, we obtain a bounded measurable $G$-equivariant cocycle
$$\fhi^* \omega\colon B^3\longrightarrow  \ell^p(\Lambda(X))$$
which is moreover alternating. By Theorem~2 in~\cite{Burger-MonodGAFA}, the fact that $B$ is an amenable $G$-measure space implies that $\fhi^* \omega$ represents a continuous bounded cohomology class $[\fhi^* \omega]$ in the space $\hbc^2(G, \ell^p(\Lambda(X)))$. The fact that $B$ is doubly ergodic with coefficients together with the alternating property of $\fhi^* \omega$ shows that this class $[\fhi^* \omega]$ vanishes only if the map $\fhi^* \omega$ vanishes almost everywhere, see~\cite{Burger-MonodGAFA} and~\cite[\S11]{Monod}.

Suppose thus for a contradiction that $\fhi^* \omega$ is a.e.\ zero. By the above description of $\omega$, this implies that $\fhi$ sends almost every triple of points in $B$ to a triple contained in some arc in $X$. We claim that this contradicts the non-elementarity, thus finishing the proof of Theorem~\ref{thm:coho}.

One way to prove this claim is as follows. Since $\fhi$ ranges in the ends of some sub-dendrite (Remark~\ref{rem:bmap}), it sends in fact almost every triple of points in $B$ to a set of at most two points. An application of Fubini's theorem now shows that $\fhi$ essentially ranges in a set of at most two points, which implies that $G$ preserves this set and therefore acts elementarily.
\end{proof}

\begin{proof}[Proof of Theorem~\ref{thm:irred}]
Suppose that an irreducible lattice $\Gamma$ in a product $G_1\times \cdots \times G_n$ of~$n\geq 2$ locally compact $\sigma$-compact groups has a non-elementary action on a dendrite $X$. Choose $1\leq p < \infty$; then $\hb^2(\Gamma, \ell^p(\Lambda(X)))$ is non-trivial by Theorem~\ref{thm:coho}. We now apply Theorem~16 from~\cite{Burger-MonodGAFA}, namely: there is an index $1\leq i \leq n$ and a non-zero closed $\Gamma$-invariant subspace $V\se \ell^p(\Lambda(X))$ such that the representation of $\Gamma$ on $V$ extends to a continuous $G$-representation which factors through the projection $G\to G_i$. (See also \S5.1 in~\cite{Burger-MonodGAFA} and Lemma~4.4 in~\cite{MonodVT}.)

Choose a non-zero vector $v\in V$. Since the non-zero level-sets of $v$ are finite and since $\Gamma$ acts non-elementarily, $\Gamma$ does not fix $v$ (we recall here that there is an equivariant map $\Lambda(X)\to X$). Thus there is a non-zero value $\lambda\neq 0$ of $v$, an element $z\in \Lambda(X)$ and $\gamma\in \Gamma$ with $v(z) = \lambda$ and $\gamma v(z) \neq \lambda$. By $p$-summability, there is $\epsilon>0$ such that no other value $\lambda'\neq\lambda$ of $v$ satisfies $|\lambda-\lambda'|<\epsilon$. Consider the disjoint subsets
$$A = \{v' :  v'(z) = \lambda \}, \kern3mm B=\{v' : |v'(z) - \lambda| \geq \epsilon\}$$
of $\ell^p(\Lambda(X))$. Both are closed, and their union contains the orbit $\Gamma v$. Therefore this union contains the extended orbit $Gv$ since $\Gamma$ projects densely to $G_i$. Since both $A$ and $B$ meet $\Gamma v$ non-trivially, this shows that $G_i$ cannot be connected.
\end{proof}

\section{Dynamics on dendrites}\label{sec:dynamics}
We start with the global picture. Let $G$ be a group acting by homeomorphisms on a compact space $X$. As in Section~\ref{sec:elem}, the space $\Prob(X)$ of Borel probability measures is compact for the weak-* topology; we observe that for any closed subspace $Y\se X$, the canonical map $\Prob(Y)\to \Prob(X)$ is an embedding.

\smallskip
Recall that the $G$-action on $X$ is called \emph{strongly proximal} if the closure of any $G$-orbit in $\Prob(X)$ contains a Dirac mass. This condition was introduced by Furstenberg~\cite{Furstenberg73} and nowadays a $G$-space that is both minimal and strongly proximal is called a (topological) \emph{$G$-boundary in the sense of Furstenberg}, not to be confused with the (measurable) Poisson--Furstenberg boundaries mentioned in Section~\ref{sec:Furstenberg}. See e.g.~\cite{GlasnerLNM} for an introduction.

\begin{thm}\label{thm:prox}
Let $G$ be a group acting on a dendrite $X$.

If the action is non-elementary and dendro-minimal, then it is strongly proximal.
\end{thm}

We immediately deduce Theorem~\ref{thm:prox-intro}; more precisely:

\begin{cor}
Let $G$ be a group with a non-elementary action on a dendrite $X$. Then the unique minimal invariant set $M$ of Lemma~\ref{lem:min} is a $G$-boundary in the sense of Furstenberg.
\end{cor}

\begin{proof}[Proof of the corollary]
Apply Theorem~\ref{thm:prox} to the dendrite $[M]$.
\end{proof}

\begin{proof}[Proof of Theorem~\ref{thm:prox}]
We shall prove that for any $x\in\Ends(X)$ and any probability measure $m\in\Prob(X)$, the Dirac mass $\delta_x$ lies in $\overline{Gm}$.

Recall the notation $U_s(t)$ introduced in Section~\ref{sec:Tits}. We shall use the fact that when $t$ is an end point and $s$ varies over a given arc abutting to $t$, the family $U_s(t)$ is a \emph{nested} basis of neighbourhoods of $t$; this follows e.g.\ from~\cite[9.3]{Nadler}.

Since the action is non-elementary, there are infinitely many end points and therefore we can choose a sequence $(y_n)$  of \emph{distinct} end points $y_n\neq x$. In particular, $m(\{y_n\})\to 0$. We now choose a sequence $(x_n)$ of regular points $x_n\in [y_n, x]$ converging to $x$. Upon extracting sub-sequences, we can assume that $[x_n,x]\subset [x_m,x]$ holds for all $n\geq m$. Finally, we choose a sequence $(z_n)$ of regular points $z_n\in[y_n,x_n]$ such that $m(U_{z_n}(y_n))\leq m(\{y_n\})+1/n$; this is possible in view of the nested neighbourhood property.

Since $z_n$ is regular, the two sub-dendrites $U_{z_n}(x)$ and $U_{z_n}(y_n)$ cover $X$; likewise for $U_{x_n}(y_n)$ and $U_{x_n}(x)$. Therefore, by Lemma~\ref{lem:dis}, there are $g_n, h_n\in G$ such that $g_n(U_{z_n}(x))\subset U_{z_n}(y_n)$ and $h_n(U_{x_n}(y_n))\subset U_{x_n}(x)$. Using $U_{z_n}(y_n) \se U_{x_n}(y_n)$, we deduce $h_ng_n(U_{z_n}(x))\subset U_{x_n}(x)$. Since $m(U_{z_n}(x))\geq1-m(\{y_n\})-1/n\to 1$ and since $(U_{x_n}(x))$ is a nested basis of neighbourhoods of $x$, we have $(h_ng_n)_*m\to\delta_x$. 
\end{proof}

We now discuss the dynamics of individual homeomorphisms of dendrites.

\begin{defn}\label{defn:ab}
Let $g$ be a homeomorphism of a dendrite $X$. We say that an arc $[x,y]\se X$ is \emph{austro-boreal} for $g$ if
$$\Fix(g)\cap[x,y]=\{x,y\}$$
wherein $\Fix(g)$ denotes the compact set of $g$-fixed points in $X$.
\end{defn}

This is a local analogue of a hyperbolic behaviour for $g$. Observe that an austro-boreal arc for $g$ is $g$-invariant and that the induced action is conjugated to an action by translations on $\RR\cup\{\pm\infty\}$. There can, of course, be several austro-boreal arcs for $g$, even infinitely many, together with non-austro-boreal $g$-invariant arcs.

\medskip

The fixed-point set $\Fix(g)$, which is always non-empty (cf.\ Lemma~\ref{lem:FP}), can be complicated. For instance, even if $X$ is just an arc, $\Fix(g)$ can be a Cantor set. The following (exclusive) alternative is a disjunction between the connectedness of this set and the presence of an austro-boreal arc somewhere in the dendrite.

\begin{lem}\label{lem:connectedFix}
Let $g$ be a homeomorphism of a dendrite $X$. Then either $\Fix(g)$ is a sub-dendrite, or there exists an austro-boreal arc for $g$. 
\end{lem}

\begin{proof}
Suppose that $g$ does not admits any austro-boreal arc. Observe that any arc between points in $\Fix(g)$ is $G$-invariant. Therefore, our assumption implies that any such arc contains a fixed point in its interior. By a minimality argument, we deduce that it contains a dense subset of fixed points and hence that this arc is fixed point-wise. It follows that $\Fix(g)$ is connected and therefore it is a sub-dendrite.
\end{proof}

The dynamics of $g$ can be described further in the presence of austro-boreal behaviour, see Proposition~\ref{prop:tectonic} below; but first, we argue that this case does indeed occur in any non-elementary group:

\begin{thm}
Let $G$ be a group with a non-elementary action on a dendrite $X$.

Then $G$ contains an element admitting an austro-boreal arc in $X$.
\end{thm}

\begin{proof}
We argue again as in the beginning of the proof of Theorem~\ref{thm:Tits}. Thus, we have an element $a\in G$ and disjoint sub-dendrites $U_p(x)$, $U_q(y)$ in $X$ such that $a(X\setminus U_p(x)) \se U_q(y)$. In particular, $a(U_q(y)) \se U_q(y)$ and hence the intersection $K_+$ of $a^n(U_q(y))$ over all $n\geq 0$ is a (non-empty) $a$-invariant sub-dendrite.

On the other hand, $a\inv (U_p(x)) \se U_p(x)$. Therefore, we deduce similarly that the intersection $K_-$ of $a^n(U_p(x))$ over all $n\leq 0$ is an $a$-invariant sub-dendrite.

We now have two $a$-invariant sub-dendrites $K_\pm$; by Lemma~\ref{lem:FP}, there must be an $a$-fixed point in each. If the fixed-point set of $a$ in $X$ were connected, it would now contain $[p,q]$ because $K_+ \se U_q(y)$ and $K_- \se  U_p(x)$. This is not possible because $a(p) \in  U_q(y)$. The conclusion now follows from Lemma~\ref{lem:connectedFix}.
\end{proof}

Let $g$ be an arbitrary homeomorphism of a dendrite $X$. If $I=[x,y]$ is an austro-boreal arc for $g$, we write $I'=I\setminus \{x,y\}$ and denote by $O(I)\se X$ the component of $X\setminus \{x,y\}$ that contains $I'$. Further, we denote by $D(g)$ the union of all $O(I)$ when $I$ ranges over all austro-boreal arcs of $g$. Finally, let $K(g)\se X$ be the complement of $D(g)$.

\smallskip
This notation provides the following tectonic decomposition:

\begin{prop}\label{prop:tectonic}
Let $X$ be a dendrite and $g$ an arbitrary homeomorphism of $X$. Then the decomposition $X=D(g) \sqcup K(g)$ has the following properties.
\begin{enumerate}[label=(\roman*)]
\item $D(g)$ is a (possibly empty) open $g$-invariant set on which $g$ acts properly discontinuously. In particular, $K(g)$ is a non-empty compact $g$-invariant set.\label{pt:tectonic:dec}
\item Every connected component of $D(g)$ is of the form $O(I)$ for some austro-boreal arc $I$, and $g$ acts co-compactly on each $O(I)$.\label{pt:tectonic:D}
\item $K(g)$ is a disjoint union of sub-dendrites of $X$. Moreover, $g$ preserves each such sub-dendrite and has a connected fixed-point set in each.\label{pt:tectonic:K}
\end{enumerate}
\end{prop}

There are at most countably many austro-boreal arcs for $g$, or equivalently, countably many components of $D(g)$. In fact, \emph{any} subset of $X$ has at most countably many components that are not reduced to a point, see e.g.~\cite[V.2.6]{Whyburn_book}. Nonetheless, $K(g)$ can have $2^{\aleph_0}$ connected components with Cantor spaces of fixed points.

\begin{proof}[Proof of Proposition~\ref{prop:tectonic}]
If $I$ is an austro-boreal arc for $g$, then $g$ preserves $O(I)$. The fact that the $g$-action on $I'$ is conjugated to a translation action on the line implies that the $g$-action on $O(I)$ is properly discontinuous and co-compact. (This is particularly apparent, for instance, if we consider the continuous retraction $O(I) \to I'$ provided by the \emph{first-point map} $X\to I$, see e.g.~10.24 and~10.25 in~\cite{Nadler}.)

Next, if $I_1$ and $I_2$ are two distinct austro-boreal arcs for $g$, then $I'_1$ and $I'_2$ are disjoint. This follows from Definition~\ref{defn:ab} because the boundary of $I'_1\cap I'_2$ in $X$ must be $g$-fixed. It further follows that $O(I_1)$ and  $O(I_2)$ are disjoint because every non-empty connected invariant subset of $O(I_i)$ must meet $I'_i$.

At this point, all statements of~\ref{pt:tectonic:dec} and~\ref{pt:tectonic:D} are justified, noting that $K(g)$ is non-empty since it contains $\Fix(g)$.

Turning to~\ref{pt:tectonic:K}, let $Y$ be a connected component of $K(g)$. Then $Y$ is closed in $X$ since $K(g)$ is so, and thus $Y$ is a sub-dendrite of $X$. Let $p\in Y$ and suppose for a contradiction that $gp\notin Y$. Then the arc $[p, gp]$ in $X$ meets $O(I)$ for some austro-boreal arc $I=[x,y]$. It follows that $[p, gp]$ contains $I$ since the boundary of $O(I)$ in $X$ is $\{x,y\}$.  Notice that $\{p,gp\}\cap\{x,y\}=\varnothing$ since $x,y$ are $g$-fixed; upon changing the order of our labels $x,y$, the point $x$ separates $p$ from $y$. Thus $gx$ separates $gp$ from $gy$, i.e.\ $x$ separates $gp$ from $y$, which contradicts our choice of the order. We have proved that $g$ preserves $Y$; now the fixed-point set of $g$ in $Y$ is connected by Lemma~\ref{lem:connectedFix}.
\end{proof}

\section{Curves}\label{sec:curves}
Many of our results for dendrites have known counterparts for group actions on the circle~$\SSS^1$. An adjustment that sometimes needs to be made for such analogies is that elementary actions on the circle should include those with finite orbits of size~$>2$. In fact, the usual notion of elementarity that allows for satisfying theorems turns out to be that the group \emph{preserves a probability measure}. This is natural since $\SSS^1$ is homogeneous (as is the Menger curve~\cite[Thm.~III]{Anderson58a}, but no other curve~\cite[Thm.~XIII]{Anderson58b}). It is also compatible with the case of dendrites by Proposition~\ref{prop:elementary}.

\medskip
However, our results certainly cannot hold for the most general curves. To begin with, every residually finite countable group (this includes all lattices of Theorem~\ref{thm:higher}) admits a \emph{free} action on the Menger curve, as follows from~\cite[Thm.~1]{Dranishnikov88}. (It is apparently unknown if this holds for all countable groups~\cite[2.20]{Bogatyi-Fedorchuk}, though faithful actions exist~\cite[Prop.~2]{Levitt04}, \cite{Kawamura}. In this context, we recall that every countable group is the full homeomorphism group of some curve~\cite{deGroot-Wille}.)

One could argue that the free action from~\cite{Dranishnikov88} preserves a probability measure since it comes from an action of the profinite completion. Yet for non-amenable groups, there is always an action without invariant probability measure on some more general curve if we relax the second countability assumption, as the following example shows. Perhaps a technical modification of that example could provide metrisable examples.

\begin{example}\label{ex:SC}
Let $G$ be a countable group. We shall assume that $G$ is finitely generated; this will be no restriction for the construction thanks to the HNN embedding theorem~\cite{HNN}. Consider the $G$-action on the topological realization $\sG$ of a locally finite Cayley graph of $G$. This extends to an action by homeomorphism on the Stone--Cech compactification $\beta\sG$ of the locally compact space $\sG$. By Prop.~5 in~\cite{Katetov50}, the compact connected space $\beta\sG$ has dimension one. It remains only to show that if $G$ preserves a probability measure on $\beta\sG$ then $G$ is amenable. The restriction map $C(\beta \sG)\cong C_{\mathrm b}(\sG)\to \ell^\infty (G)$ admits a $G$-equivariant (linear, unital, positive) right inverse $\ell^\infty (G)\to C_{\mathrm b}(\sG)$ given by extending functions affinely on the edges. Therefore, the probability measure provides an invariant mean on $\ell^\infty (G)$.
\end{example}

It turns out that there is a setting to which most of our results can be extended, namely \emph{local dendrites}. Although they can be defined abstractly as curves that are absolute neighbourhoods retracts, the key property for us is that they contain at most finitely many embedded copies of $\SSS^1$, see~\cite[\S46 VII]{Kuratowski_T2}. We deduce:

\begin{lem}\label{lem:local}
Let $G$ be a group acting by homeomorphisms on a local dendrite $X$. Then either $X$ is a dendrite or $G$ admits a finite index subgroup preserving a circle in $X$.\qed
\end{lem}

Here is a case where no work at all is needed to combine our results on dendrites with results for the circle:

\begin{proof}[Proof of Corollary~\ref{cor:higher:loc}]
Suppose that $\Gamma$ acts on a local dendrite $X$. If $X$ is a dendrite, then we are done by Theorem~\ref{thm:higher}. Otherwise, let $\Gamma' < \Gamma$ be a finite index subgroup as in Lemma~\ref{lem:local}. The group $\Gamma'$ is still a lattice in the ambient simple algebraic group and therefore we can apply the theorem of~\cite{Ghys99,Burger-Monod1} which states that any $\Gamma'$-action on the circle has a finite orbit. It follows that $\Gamma$ also has a finite orbit in $X$.
\end{proof}

We now turn to an extension of Theorem~\ref{thm:coamen}. This time we need the broader notion of elementarity discussed above, as for instance even the $\SSS^1$-action on itself illustrates.

\begin{thm}\label{thm:coamen:S1}
Let $G$ be a group admitting two commuting co-amenable subgroups.

Then any  $G$-action on a local dendrite preserves a probability measure.
\end{thm}

Although the following lemma is very simple, we warn the reader that the statement would fail if $G'$ were merely supposed co-amenable in $G$, even with $G'$ normal in $G$ and containing $H$, see~\cite{Monod-Popa}.

\begin{lem}\label{lem:coamen}
Let $G$ be a group, $H<G$ a co-amenable subgroup and $G'<G$ a subgroup of finite index. Then $G'\cap H$ is co-amenable in $G'$.
\end{lem}

\begin{proof}[Proof of the lemma]
One of the equivalent characterisation of co-amenability is the existence of a $G$-invariant mean $\mu$ on the coset space $G/H$~\cite{Eymard72}. We can realize $G'/(G'\cap H)$ as a $G'$-invariant subset of $G/H$. If $\mu$ gives positive mass to this subset, then we can renormalize $\mu$ to witness the co-amenability of $G'\cap H$ in $G'$. Otherwise, we reach a contradiction because $G/H$ can be covered by finitely many $G$-translates of $G'/(G'\cap H)$ using coset representatives for $G'$ in $G$.
\end{proof}

\begin{proof}[Proof of Theorem~\ref{thm:coamen:S1}]
Let $X$ be a local dendrite with a $G$-action.
In view of Theorem~\ref{thm:coamen}, we can assume that $X$ contains a simple closed curve. By Lemma~\ref{lem:coamen}, we can assume that $G$ preserves such a curve by replacing $G$ with the stabiliser of a curve and the subgroups accordingly. Therefore it remains only to consider the case where $X$ is the circle.

Let $H_1, H_2<G$ be two commuting co-amenable subgroups. We suppose for a contradiction that $G$ does not preserve a probability measure on $X$, and therefore neither $H_1$ nor $H_2$ do. Recall that each $H_i$ admits a unique minimal closed non-empty invariant subset $M_i\se X$, see e.g.~\cite[Thm.~2.1.1]{Navas_book}. By minimality and since the $H_i$ commute, we have $M_1=M_2$. For the purpose of reaching a contradiction, we may assume that $M_i=X$ after passing to a circle quotient, see p.~64 in~\cite{Navas_book} for the fact that an invariant measure on the quotient would lift. Now the action of any given element $h\in H_1$ can be conjugated to a rotation since it commutes to the group $H_2$ acting minimally, and at least some $h\in H_1$ is not of finite order (e.g.\ by applying Margulis' alternative~\cite{MargulisS1} to $H_1$). In particular this element $h$ has a unique invariant probability measure on $X\cong \SSS^1$, the Haar measure, since it generates a dense subgroup of $\SSS^1$. By uniqueness, $H_2$ preserves this measure, which is a contradiction.
\end{proof}

The cohomological obstruction of Theorem~\ref{thm:coho-intro} can be extended as follows.

\begin{cor}
Let $G$ be a group such that $\hb^2(G, V)=0$ for every unitary representation $V$.

Then any $G$-action on a local dendrite preserves a probability measure.
\end{cor}

\begin{proof}
The cohomological assumption made on $G$ is preserved by passing to finite index subgroups. Indeed, this follows by the appropriate version of cohomological induction, see~\cite[\S10.1]{Monod}; for subgroups of finite index, the induction modules used in bounded cohomology preserve unitarity. Therefore, we can again consider separately the cases when $G$ acts on a dendrite and when $G$ acts on the circle. In the first case, the statement follows from Theorem~\ref{thm:coho-intro}.

In the latter case, we can furthermore assume that $G$ preserves the orientation of the circle by passing again to a finite index subgroup. We consider the bounded Euler class~\cite{Ghys_EM} in $\hb^2(G, \ZZ)$. Since its image in $\hb^2(G, \RR)$ vanishes, the $G$-action is quasi-conjugated to an action by rotations (see e.g.~\cite[3.2]{Burger_UC}). Once again, we conclude as in~\cite[p.~64]{Navas_book} that $G$ preserves a probability already before quasi-conjugation.
\end{proof}

Finally, combining the Tits alternative of Theorem~\ref{thm:Tits} with Margulis' Tits alternative~\cite{MargulisS1}, we obtain:

\begin{cor}\label{cor:Tits:S1}
Let $G$ be a group acting on a local dendrite.

Then either $G$ contains a non-abelian free subgroup or it preserves a probability measure.
\end{cor}

\begin{proof}
Let $X$ be a local dendrite with a $G$-action. If $X$ is a dendrite, we apply Theorem~\ref{thm:Tits}. Otherwise, we apply Theorem~3 from~\cite{MargulisS1} to the subgroup $G'<G$ of Lemma~\ref{lem:local} acting on a circle in $X$. This result provides either a non-abelian free subgroup of $G'$, hence of $G$, or a probability measure $\mu$ on the circle preserved by $G'$. We can regard $\mu$ as a $G'$-fixed measure on $X$, and now the average over $G/G'$ of its $G$-translates provides a $G$-invariant probability measure on $X$.
\end{proof}

\section{Further considerations}\label{sec:further}
It is not difficult to design (non-trivial) dendrites without any homeomorphism, see e.g.\ p.~443 in~\cite{deGroot-Wille}. Much more strikingly, dendrites were constructed that are not homeomorphic to any subset of themselves~\cite{Miller32} and are even \emph{chaotic}, or \emph{totally heterogeneous}, see~\cite[\S3]{Besicovitch45}. (For the complexity of dendrite homeomorphisms and embeddings, see~\cite[Thm.~6.7]{Camerlo-Darji-Marcone} respectively~\cite{Marcone-Rosendal}.)

\medskip
However, other dendrites admit such a profusion of homeomorphisms that it seems impossible to associate any rigid structure to them (in contrast to Bowditch's non-nesting actions on dendrites). We shall illustrate this on \emph{universal dendrites}.

\medskip
Wa\.zewsi's universal dendrite $D_\infty$, introduced in~\cite[p.~9]{WazewskiPHD}, \cite[p.~57]{Wazewski23} with the notation $D^*$, has the following properties. Every branch point has infinite order, $\Br(D_\infty)$ is dense in $D_\infty$, and every dendrite can be embedded into $D_\infty$.

There are similar constructions of universal dendrites $D_n$ whose branch points all have order $n\in \NN$ and variants where several orders are allowed, see e.g.~\cite[\S6]{Charatonik-Dilks}.

For any $3\leq n\leq \infty$, the action of the homeomorphism group $\Homeo(D_n)$ on $D_n$ is non-elementary and dendro-minimal. In fact, much more is true:

\begin{prop}\label{prop:oligo}
The action of $\Homeo(D_n)$ on $\Ends(D_n)$ is oligomorphic.
\end{prop}

Recall that oligomorphy~\cite{Cameron_oligo} means that for each $p\in\NN$ there are finitely many $\Homeo(D_n)$-orbits for the diagonal action on $\Ends(D_n)^p$.

\smallskip
It follows that there is a finite number of $p$-tuples in $\Ends(D_n)^p$ such that the union of their orbits is dense in $(D_n)^p$, because $\Ends(D_n)$ is dense in $D_n$. In particular, this constitutes a strong negation of the convergence action property of Bowditch's setting and certainly ruins any hope for a meaningful geometric interpretation of the  $\Homeo(D_n)$-action on $D_n$.

We will establish stronger statements in that direction in the forthcoming article~\cite{DM_dendrites_2}.

\begin{proof}[Proof of Proposition~\ref{prop:oligo}]
Given $x_1, \ldots, x_p \in\Ends(D_n)$, consider the compact tree $[\{x_1, \ldots, x_p \}]$ spanned in $D_n$. We claim that the combinatorial type of this tree is a complete invariant for the diagonal action of $\Homeo(D_n)$, which establishes the proposition.

Following~\cite{Charatonik91}, \cite{Charatonik95}, we denote for any distinct $x,y\in D_n$ by $D_n(x,y)$ the closure of the component of $D_n \setminus\{x,y\}$ containing the interior of the arc $[x,y]$. (This coincides with the closure of the set $O([x,y])$ in the notation of Proposition~\ref{prop:tectonic}.) Then $D_n(x,y)$ is homeomorphic to $D_n$ and this homeomorphism can be chosen so as to send $x,y$ to any pair of ends of $D_n$, see Proposition~4.1 in~\cite{Charatonik91}. Now the claim follows by decomposing $D_n$ into the different sub-dendrites $D_n(x,y)$ obtained from all adjacent nodes (including leaves) of the tree $[\{x_1, \ldots, x_p \}]$ and pasting the corresponding homeomorphisms together.

This is exactly the argument used in~\cite[Prop.~4.3]{Charatonik91} for the transitivity on distinct triples in $\Ends(D_n)$ (the case $n=3$ was previously established in~\cite{Kato87} and more general dendrites were treated in~\cite{Charatonik95}). We refer to~\cite{Charatonik91}, \cite{Charatonik95} for more details.
\end{proof}

A completely opposite case arises from compactifying simplicial trees. Consider a simplicial tree $T$ and denote by $\sT$ it geometric realization. Let $\partial \sT$ be its ideal boundary (in the CAT(0) sense~\cite[\S II.8]{Bridson-Haefliger}). There is a compact Hausdorff topology on $\overline \sT = \sT \sqcup \partial\sT$ with sub-basis given by the connected components of complements of points; see e.g.~\cite[\S1]{Monod-Shalom1}. When $T$ is locally finite, this is the usual cone topology, which is none other than Freudenthal's construction~\cite{Freudenthal31}, \cite{Freudenthal45}; but in general, $\partial \sT$ is not closed in $\overline \sT$.

\begin{prop}\label{prop:vegetale}
For any countable simplicial tree, $\overline \sT$ is a dendrite such that any two branch points are separated by at most finitely many branch points.

Conversely, suppose that $X$ is a dendrite satisfying this finiteness condition. Then there exists a canonical countable simplicial tree $T_X$ with a canonical homeomorphism $\overline{\sT_X}\cong X$.
\end{prop}

\noindent
(Since we do not use Proposition~\ref{prop:vegetale}, we only sketch its straightforward proof below.)

\medskip

This correspondence is not quite bijective because the passage $T\mapsto \sT$ erases degree two vertices. Nonetheless, the canonical aspect of Proposition~\ref{prop:vegetale} highlights the more limited nature of the homeomorphism group of such dendrites, since it leads to a decomposition
$$\Homeo(X) \ \cong\  \Fix (V)\rtimes \Aut(T_X) $$
where $\Fix (V)$ denotes the fixator in $\Homeo(X)$ of the vertex set $V$ of $T_X$ (viewed as a subset of $X$).

\medskip
Without canonicality, a metrisation as above would have little interest. Recall that any dendrite can be metrised to become an $\RR$-tree. This can be deduced already from~\cite[\S9]{Menger28}; see also~\cite[\S12]{Kuratowski-Whyburn}, or~\cite[Thm.~4]{Moise49} and~\cite[Thm.~8]{Bing49} for the ultimate generalisation. Conversely, an $\RR$-tree can be equipped with a weak topology (or better, uniform structure) and compactified, becoming a dendrite~\cite{Coulbois-Hilion-Lustig} provided it was separable. This weak topology has also been called the \emph{observer's topology} in~\cite{Coulbois-Hilion-Lustig} and is the \emph{convex topology} $\mathcal{T}_c$ of~\cite{MonodCAT0}; cf.\ also~\cite[\S5.1]{Favre-Jonsson}.

This metrisation has no bearing on the study of the homeomorphism group unless there is at least some metric restriction on the dynamics, such as the \emph{non-nesting} condition of Bowditch --- which utterly lacks in the situation described in Proposition~\ref{prop:oligo}.

\begin{proof}[Sketch of proof for Proposition~\ref{prop:vegetale}]
Notice that the weak topology coincides with the ordinary topology on any arc in $\sT$. In particular, $\overline \sT$ is a continuum; criterion~\ref{pt:sep} of Section~\ref{sec:prelim} makes it easy to check that it is a dendrite. The finiteness condition is immediate.

Conversely, let $X$ be dendrite satisfying that finiteness condition. Consider those ends of $X$ that are not limits of branch points; there are at most countably many such ends. We define the vertex set $V$ of the tree $T_X$ as the union of this subset of ends with $\Br(X)$. We declare that a pair $v\neq v'$ in $V$ forms an edge if $v$ and $v'$ cannot be separated by a branch point. The resulting graph $T_X$ is connected thanks to the finiteness condition (and to the fact that $X$ is arcwise connected). This graph is acyclic because $X$ is a dendrite.

Finally, the inclusion map $V\to X$ can be extended on the edges to yield a map $\sT_X\to X$ which extends to a homeomorphism $\overline{\sT_X}\cong X$ for the weak topology.
\end{proof}

\bibliographystyle{halpha}
\bibliography{dendrite}
\end{document}